\documentclass[12pt]{article}
\usepackage{setspace} 
\usepackage[dvips]{graphicx} 
\usepackage{amsmath,amsthm,amsfonts}
\usepackage{epsfig}
\RequirePackage[colorlinks,citecolor=blue,urlcolor=blue,linkcolor=blue]{hyperref}
\usepackage{graphicx,xspace,colortbl}
\usepackage{epstopdf}
\usepackage[titletoc,title]{appendix}
\usepackage[dvipsnames]{xcolor}
\usepackage{amsmath,amsthm,amsfonts,amssymb, tabularx, booktabs}
\usepackage{color}
\usepackage{enumerate}
\usepackage{fancybox}
\usepackage{epsfig}
\usepackage{pdfsync}
\usepackage{array}
\usepackage{float}
\usepackage{caption}
\usepackage{subcaption}
    \def\re{\textnormal {Re}}
    
    \def\r{{\mathbb R}}

	\newtheorem{theorem}{Theorem}
	\newtheorem{lemma}{Lemma}
	\newtheorem{proposition}{Proposition}
	
	\newtheorem{definition}{Definition}

\title{Lagrange Inversion Theorem for Dirichlet series}
\author{
{Alexey Kuznetsov
\footnote{Dept. of Mathematics and Statistics,  York University,
4700 Keele Street, Toronto, ON, M3J 1P3, Canada.  \newline
E-mail:  akuznets@yorku.ca  } 
 }}
 
 \date{\today}

\begin{document}
\maketitle

\begin{abstract} 
We prove an analogue of the Lagrange Inversion Theorem for Dirichlet series. The proof is based on studying properties of Dirichlet convolution polynomials, which are analogues of convolution polynomials introduced by Knuth in \cite{Knuth1992}. 
\end{abstract}

{\vskip 0.15cm}


\section{Introduction}\label{section_intro}


The Lagrange Inversion Theorem states that if $\phi(z)$ is analytic near $z=z_0$ 
with $\phi'(z_0)\neq 0$ then the function $\eta(w)$, defined as the solution to $\phi(\eta(w))=w$, is analytic near $w_0=\phi(z_0)$ and is represented there by the power series 
\begin{equation}
\eta(w)=z_0+\sum\limits_{n\ge 0} \frac{b_n}{n!} (w-w_0)^n,
\end{equation}
where the coefficients are given by
\begin{equation}\label{Lagrange_1}
b_n=\lim\limits_{z\to z_0} 
\frac{{\textrm {d}}^{n-1}}{{\textrm{d}}z^{n-1}} \Big(\frac{z-z_0}{\phi(z)-\phi(z_0)} \Big)^n.
\end{equation}

This result can be stated in the following simpler form. We assume that $\phi(z_0)=z_0=0$ and $\phi'(0)=1$. We define $A(z)=z/\phi(z)$ and check that $A$ is analytic near $z=0$ and is represented there by the power series 
\begin{equation}\label{Lagrange_Burmann1}
A(z)=1+\sum\limits_{n\ge 1} a_n z^n. 
\end{equation}
Let us define $B(z)$ implicitly via 
\begin{equation}\label{def_B_z}
A(zB(z))=B(z).
\end{equation}
 Then the function $\eta(z)=zB(z)$ solves equation $\eta(z)/A(\eta(z))=z$, which is equivalent to
$\phi(\eta(z))=z$. Applying Lagrange Inversion Theorem in the form \eqref{Lagrange_1} we  obtain 
\begin{equation}\label{Lagrange_Burmann2}
B(z)=\sum\limits_{n\ge 0} \frac{a_n(n+1)}{n+1} z^n,
\end{equation}
where the coefficients $a_n(k)$ are defined via 
 \begin{equation}\label{Lagrange_Burmann3}
 A(z)^k=\sum\limits_{n\ge 0} a_n(k) z^n, \;\;\; k\in {\mathbb N}.  
 \end{equation}
 The above result is known as the Lagrange-B\"urmann formula. It is usually stated as a corollary of the Lagrange Inversion Theorem, but in fact this result is easily seen to be equivalent to the Lagrange Inversion Theorem, as formula \eqref{Lagrange_1} follows from 
formulas \eqref{Lagrange_Burmann1}-\eqref{Lagrange_Burmann3} by going back to functions $\phi(z)=z/A(z)$ and $\eta(z)=zB(z)$ and performing affine transformation of variables. 

For our purposes we would need yet another form of the Lagrange Inversion Theorem. Again, we assume that $A(z)$ is given by the power series \eqref{Lagrange_Burmann1}. We fix $w\in {\mathbb C}$ and define $B(z)$ implicitly 
via 
\begin{equation}\label{def_B_z_w}
A(z B(z)^w)=B(z).
\end{equation}
 It is easy to see that the function $B(z)$ is analytic near $z=0$ and satisfies $B(0)=1$. It turns out that for any $x\in {\mathbb C}$ we have Taylor series expansion 
\begin{equation}\label{formula_Bzwx}
B(z)^x=\sum\limits_{n\ge 0} \frac{x}{x+nw} a_n(x+nw) z^n, 
\end{equation}
where $a_n(x)$ are defined via 
 \begin{equation}\label{def_a_n_x}
 A(z)^x=\sum\limits_{n\ge 0} a_n(x) z^n, \;\;\; x\in {\mathbb C}.  
 \end{equation}
The above result does not seem to be well-known. However, it is not hard to derive it from the Lagrange Inversion Theorem, see 
\cite{ScottSokal2009} and \cite{Stanley}[Exercise 5.58]. 

Formulas \eqref{def_B_z_w}-\eqref{def_a_n_x} invite a number of observations. First of all, formula \eqref{def_a_n_x} implies 
 that  $a_n(x)$ are polynomials in $x$ of degree at most $n$ and that they satisfy the convolution identity
\begin{equation}\label{convolution_pols}
a_n(x+y)=\sum\limits_{k=0}^n a_k(x) a_{n-k}(y), \;\;\; x,y\in {\mathbb C}.
\end{equation}
Thus $a_n(x)$ are {\it  convolution polynomials} as defined by Knuth in \cite{Knuth1992} (see also the paper \cite{Zeng1996} by Zeng  for multi-index extensions of such polynomial families). 
Formula \eqref{formula_Bzwx} implies that (for fixed $w$) the functions
\begin{equation}\label{def_b_n_pols}
b_n(x):=\frac{x}{x+nw} a_n(x+nw)
\end{equation}
 also form a convolution polynomial family (it is easy to see that $b_n(x)$ are indeed polynomials since formula 
 \eqref{def_a_n_x} implies that $a_n(0)=0$, thus $a_n(x)/x$ is a polynomial for every $n\ge 1$). 

Our goal in this paper is to prove a Dirichlet series analogue of the Lagrange Inversion Theorem in the form  \eqref{def_B_z_w}-\eqref{def_a_n_x}. This will be achieved in two steps. First, in Section \ref{section_polynomials} we will introduce and study the {\it Dirichlet convolution polynomials}, which are defined by an identity similar to \eqref{convolution_pols} but with the additive convolution replaced by the Dirichlet convolution. The main result of Section \ref{section_polynomials} is Theorem \ref{thm1}, which gives a construction of new Dirichlet convolution 
polynomials that is similar in spirit to the one given in equation \eqref{def_b_n_pols}. 
Armed with these results, in Section \ref{section_inversion} we will prove the Lagrange Inversion Theorem for Dirichlet series.


\section{Dirichlet convolution polynomials}\label{section_polynomials}


We start by defining Dirichlet analogues of convolution polynomials.
\begin{definition}\label{definition1}
{\textnormal{Polynomials $\{\alpha_n(x)\}_{n\ge 1}$ satisfying 
\begin{equation}\label{alpha_n_mult_convolution}
\alpha_n(x+y)=\sum\limits_{d|n} \alpha_d(x) \alpha_{\frac{n}{d}}(y), \;\;\; n\ge 1, \;\;\; x,y \in {\mathbb R}. 
\end{equation} 
 will be called {\it Dirichlet convolution polynomials}}}. 
\end{definition}

In the next result we establish that, under some mild conditions, any functions $\alpha_n(x)$ that satisfy equations 
\eqref{alpha_n_mult_convolution} must necessarily be polynomials. Thus, when studying Dirichlet convolution equations of the form
\eqref{alpha_n_mult_convolution}, there is no need to look beyond polynomial solutions.  
We recall that the arithmetic function $\Omega(n)$ counts the number of (not necessarily distinct) prime factors of $n$. In other words,  if $n=p_1^{j_1}p_2^{j_2}\dots p_k^{j_k}$ for prime numbers $p_i$ then $\Omega(n)=j_1+j_2+\dots+j_k$.

\begin{proposition}\label{proposition_uniqueness}
Let $\alpha_1(x)\equiv 1$ and $\{\alpha_n(x)\}_{n\ge 2}$ be continuous complex-valued functions of $x\in {\mathbb R}$ that satisfy the Dirichlet convolution identity \eqref{alpha_n_mult_convolution}.
Then for $n\ge 2$
\begin{equation}\label{alpha_n_formula1}
\alpha_n(x)=\sum\limits_{k=1}^{\Omega(n)} {{x}\choose{k}} 
\sum\limits_{\substack{n=d_1d_2\cdots d_k\\ d_i \ge 2}} \alpha_{d_1}(1)\alpha_{d_2}(1)\cdots \alpha_{d_k}(1).
\end{equation}
In particular, for every $n\ge 2$ the function $\alpha_n(x)$ is a polynomial of degree at most 
$\Omega(n)$. 
 \end{proposition}
  \begin{proof}
   Let us denote by $\phi(s)$ the formal Dirichlet series
 \begin{equation}\label{def_f}
 \phi(s)=1+\sum\limits_{n\ge 2} \frac{\alpha_n(1)}{n^s}. 
 \end{equation}
 Define $\{c_n(x)\}_{n\ge 1}$ via  
 \begin{equation}\label{def_c_n(x)}
 \phi(s)^x=\sum\limits_{n\ge 1} \frac{c_n(x)}{n^s}
 \end{equation}
Here by $\phi(s)^x$ we understand the formal Dirichlet series obtained by applying binomial series in the following way: 
 \begin{align*}
 \phi(s)^x&= \Big(1+\sum\limits_{n\ge 2} \frac{\alpha_n(1)}{n^s} \Big)^x=1+ \sum\limits_{k\ge 1} {{x}\choose{k}} 
 \Big(\sum\limits_{n\ge 2} \frac{\alpha_n(1)}{n^s} \Big)^k.
 \end{align*}
 Collecting the terms in front of $n^{-s}$ we get
 $$
 \Big(\sum\limits_{n\ge 2} \frac{\alpha_n(1)}{n^s} \Big)^k=
 \sum\limits_{n\ge 2} \frac{1}{n^s} \sum\limits_{\substack{n=d_1d_2\cdots d_k\\ d_i \ge 2}} \alpha_{d_1}(1)\alpha_{d_2}(1)\cdots \alpha_{d_k}(1).
 $$
 Note that 
 $$
 \sum\limits_{\substack{n=d_1d_2\cdots d_k\\ d_i \ge 2}} \alpha_{d_1}(1)\alpha_{d_2}(1)\cdots \alpha_{d_k}(1)=0
 $$
 for $k> \Omega(n)$:  we have an empty sum since we can not write $n$ as a product of $k>\Omega(n)$ factors $d_1, d_2, \dots, d_k$, each factor satisfying $d_i>2$. 
 Combining the above formulas we derive an expression 
 $$
 \phi(s)^x=1+\sum\limits_{n\ge 2} \frac{1}{n^s} \times \Big[ \sum\limits_{k=1}^{\Omega(n)} {{x}\choose{k}} 
\sum\limits_{\substack{n=d_1d_2\cdots d_k\\ d_i \ge 2}} \alpha_{d_1}(1)\alpha_{d_2}(1)\cdots \alpha_{d_k}(1)\Big]. 
 $$ 
 Thus we have $c_1(x)\equiv 1$ and for $n\ge 2$ 
 $$
 c_n(x)=\sum\limits_{\substack{n=d_1d_2\cdots d_k\\ d_i \ge 2}} \alpha_{d_1}(1)\alpha_{d_2}(1)\cdots \alpha_{d_k}(1).
 $$
It remains to prove that $\alpha_n(x)=c_n(x)$ for all $n\ge 1$ and $x\in {\mathbb R}$. The proof will proceed by induction. 
 
 We know that $\alpha_1(x)=c_1(x)=1$ for all $x\in {\mathbb R}$. Assume that $\alpha_k(x)=c_k(x)$ for all $k\le n-1$ and $x\in {\mathbb R}$. 
 Both $c_n(x)$ and $\alpha_n(x)$ satisfy the Dirichlet convolution identity \eqref{alpha_n_mult_convolution}, thus
 $$
 c_n(x+y)=c_n(x)+c_n(y)+\sum\limits_{\substack{2\le d \le n-1\\ d|n}} c_d(x) c_{\frac{n}{d}}(y). 
 $$ 
 By the induction hypothesis we have $c_k(x)=\alpha_k(x)$ for all $k\le n-1$ and $x\in {\mathbb R}$, therefore we can rewrite the above identity in the form
 $$
 c_n(x+y)=c_n(x)+c_n(y)+\sum\limits_{\substack{2\le d \le n-1\\ d|n}} \alpha_d(x) \alpha_{\frac{n}{d}}(y).
 $$ 
 Moreover, we have a similar equation for $\alpha_n(x+y)$:
 $$
 \alpha_n(x+y)=\alpha_n(x)+\alpha_n(y)+\sum\limits_{\substack{2\le d \le n-1\\ d|n}} \alpha_d(x) \alpha_{\frac{n}{d}}(y).
 $$ 
 Thus the function $h(x)=c_n(x)-\alpha_n(x)$ satisfies $h(x+y)=h(x)+h(y)$, $h$ is continuous on ${\mathbb R}$ and 
 $h(1)=0$ (since $c_n(1)=\alpha_n(1)$, see \eqref{def_c_n(x)}). We conclude that $h(x)$ is identically equal to zero, so that $\alpha_n(x)\equiv c_n(x)$. 
 \end{proof}

An important observation is that formula \eqref{alpha_n_formula1} implies that $\alpha_n(0)=0$ for all $n\ge 2$. Thus 
the functions 
\begin{equation}\label{def_hat_alpha_n}
\hat \alpha_n(x):=x^{-1} \alpha_n(x), \;\;\; n\ge 2
\end{equation} 
are polynomials of degree $\Omega(n)-1$. In particular, for every prime $p$ we have $\hat \alpha_p(x)\equiv {\textrm {const}}$. 

In the next proposition we collect a number of properties of Dirichlet convolution polynomials. 
\begin{proposition}\label{proposition_properties_alpha_n} Let $\alpha_n(x)$ be Dirichlet convolution polynomials. 
\begin{itemize}
\item[(i)] For $n\ge 2$, $x,y\in {\mathbb C}$ and $y\neq 0$ 
\begin{equation}\label{alpha_n_formula1.5}
\alpha_n(x)=\sum\limits_{k=1}^{\Omega(n)} {{x/y}\choose{k}} 
\sum\limits_{\substack{n=d_1d_2\cdots d_k\\ d_i \ge 2}} \alpha_{d_1}(y)\alpha_{d_2}(y)\cdots \alpha_{d_k}(y).
\end{equation}
\item[(ii)] For $n\ge 2$ and $x\in {\mathbb C}$ 
\begin{equation}\label{alpha_n_formula2}
\alpha_n(x)=\sum\limits_{k=1}^{\Omega(n)} \frac{x^k}{k!} 
\sum\limits_{\substack{n=d_1d_2\cdots d_k\\ d_i \ge 2}} \hat \alpha_{d_1}(0)\hat \alpha_{d_2}(0)\cdots \hat \alpha_{d_k}(0).
\end{equation}
\item[(iii)] For $N\ge 2$ and $s,x\in {\mathbb C}$ 
\begin{equation}\label{e_Dirichlet_polynomial}
\exp\Big( x \sum\limits_{n=2}^N \frac{\hat \alpha_n(0)}{n^s} \Big)=
\sum\limits_{n=1}^N \frac{\alpha_n(x)}{n^s} + \phi_{N+1}(s)
\end{equation}
for some absolutely convergent Dirichlet series $\phi_{N+1}(s)=\sum_{n\ge N+1} c_n(x)n^{-s}$.
\item[(iv)] For $n\ge 2$ and $x\in {\mathbb C}$ 
\begin{equation}\label{identity_with_integrals}
\alpha_n(x)=\sum\limits_{d|n,d \ge 2} \hat \alpha_d(0) \int_0^x \alpha_{\frac{n}{d}}(y) {\textrm{d}} y.
\end{equation} 
\item[(v)] For $n\ge 2$ and $x\in {\mathbb C}$ 
\begin{equation}\label{convolution_eqn}
\ln(n)  \hat \alpha_n(x)=\sum\limits_{d|n, d\ge 2} \ln(d) \hat \alpha_d(0) \alpha_{\frac{n}{d}}(x). 
\end{equation}
\item[(vi)] If the arithmetic function $n  \mapsto \alpha_n(x)$ is multiplicative for some nonzero $x \in {\mathbb C}$ (that is, 
$\alpha_{mn}(x)=\alpha_m(x)\alpha_n(x)$ for all relative prime positive integers $m$ and $n$), then it is multiplicative for all $x \in {\mathbb C}$. 
\end{itemize}
\end{proposition}
\begin{proof}
In the proof of Proposition \ref{proposition_uniqueness} we saw that if $\alpha_n(x)$ are 
Dirichlet convolution polynomials then the corresponding formal Dirichlet series satisfy an identity
\begin{equation}\label{bijection}
1+\sum\limits_{n\ge 2} \frac{\alpha_n(x)}{n^s}  
=\Big(1+\sum\limits_{n\ge 2} \frac{\alpha_n(y)}{n^s} \Big)^{\frac{x}{y}},
\end{equation}
where $x,y \in {\mathbb C}$ and $y\neq 0$. 
Expanding the right-hand side in binomial series and collecting the terms in front of $n^{-s}$ we obtain identity \eqref{alpha_n_formula1.5}.

Formula \eqref{alpha_n_formula2} follows from \eqref{alpha_n_formula1.5} by taking the limit 
$y\to 0$ and using the fact that 
$$
{{x/y}\choose{k}}y^k=\frac{x(x-y)(x-2y)\dots (x-(k-1)y)}{k!}  \to \frac{x^k}{k!}. 
$$

Formula \eqref{e_Dirichlet_polynomial} follows by expanding 
$$
\exp\Big( x \sum\limits_{n=2}^N \frac{\hat \alpha_n(0)}{n^s} \Big)=1+\sum\limits_{k\ge 1} 
\frac{x^k}{k!} \Big( \sum\limits_{n=2}^N \frac{\hat \alpha_n(0)}{n^s} \Big)^k,
$$
 collecting the coefficients in front of 
$n^{-s}$ and comparing the resulting coefficients with expression for $\alpha_n(x)$ given in 
\eqref{alpha_n_formula2}.

To prove \eqref{identity_with_integrals} we rearrange the convolution identity \eqref{alpha_n_mult_convolution} in the form
$$
\frac{1}{y}(\alpha_n(x+y)-\alpha_n(x))=\sum\limits_{d|n,d \ge 2} \frac{\alpha_d(y)}{y} \alpha_{\frac{n}{d}}(x). 
$$
Taking the limit as $y\to 0$ we obtain
$$
\frac{{\textrm{d}}}{{\textrm{d}}x} \alpha_n(x)=\sum\limits_{d|n,d \ge 2} \hat \alpha_d(0) \alpha_{\frac{n}{d}}(x). 
$$
Formula \eqref{identity_with_integrals} follows from the above identity and the fact that $\alpha_n(0)=0$. 

Next, let us prove formula \eqref{convolution_eqn}. We fix a large integer $N$, take derivative in $s$ of both sides of \eqref{e_Dirichlet_polynomial} and obtain
\begin{align*}
-\sum\limits_{n=2}^{N} \ln(n) \frac{\alpha_n(x)}{n^s}+\phi'_{N+1}(s)&=
-\exp\Big(x \sum\limits_{n=2}^N \frac{\hat \alpha_n(0)}{n^s} \Big)
x \sum\limits_{n=2}^N \ln(n)\frac{\hat \alpha_n(0)}{n^s}\\
&=-\Big(\sum\limits_{n=2}^{N}  \frac{\alpha_n(x)}{n^s}+\phi_{N+1}(s)\Big)
x\sum\limits_{n=2}^N \ln(n)\frac{\hat \alpha_n(0)}{n^s}
\end{align*}
Comparing the coefficients in front of $n^{-s}$ in the above identity and using the fact that both Dirichlet series $\phi_{N+1}(s)$ and $\phi'_{N+1}(s)$ only have terms $m^{-s}$ with $m>N$ we obtain formula \eqref{convolution_eqn}
for all $2\le n \le N$. Since $N$ was arbitrary, this proves that \eqref{convolution_eqn} holds true for all $n\ge 2$. 

Finally, let us prove item (vi). Assume that the arithmetic function $n\mapsto \alpha_n(x)$ is multiplicative. This is equivalent to saying that the formal Dirichlet series associated to 
$\alpha_n(x)$ can be written as a product of series over primes: 
\begin{equation}
\sum\limits_{n\ge 1} \frac{\alpha_n(x)}{n^s}  =
\prod\limits_{p} \Big(1+ \sum\limits_{j\ge 1} \frac{c(p,j)}{p^{js}} \Big).
\end{equation}
 Then, applying formula
\eqref{bijection} we see that for every $y\in {\mathbb R}$ 
\begin{align*}
\sum\limits_{n\ge 1} \frac{\alpha_n(y)}{n^s} &=
\Big(\sum\limits_{n\ge 1} \frac{\alpha_n(x)}{n^s} \Big)^{\frac{y}{x}}=
\prod\limits_{p} \Big(1+ \sum\limits_{j\ge 1} \frac{c(p,j)}{p^{js}} \Big)^{\frac{y}{x}}
=\prod\limits_{p} \Big(1+ \sum\limits_{j\ge 1} \frac{\tilde c(p,j)}{p^{js}} \Big),
\end{align*}
for some other coefficients $\tilde c(p,j)$. Thus the function $n\mapsto \alpha_n(y)$ is also multiplicative. 
\end{proof}

Formula \eqref{e_Dirichlet_polynomial} shows that Dirichlet convolution polynomials $\alpha_n(x)$ are obtained from the formal Dirichlet series $f(s)=\sum_{n\ge 2} \hat a_n(0) n^{-s}$ via identity
$$
e^{x f(s)}=\sum\limits_{n\ge 1} \frac{\alpha_n(x)}{n^s}. 
$$
We will call $f$ {\it the generating series} for $\alpha_n(x)$.

Next we consider the question of how to construct new Dirichlet convolution polynomials starting from existing ones. The following observation follows easily from Definition \ref{definition1}:

\vspace{0.25cm}
\noindent
Assume that $\alpha_n(x)$ and $\beta_n(x)$ are Dirichlet convolution polynomials. Then we can construct new Dirichlet convolution polynomials $\gamma_n(x)$ in one of the following three ways
\begin{itemize}
\item[(i)]  $\gamma_n(x)=\alpha_n(wx)$ for some $w\in {\mathbb C}$ and all $n\ge 1$;
\item[(ii)] 
$
\gamma_n(x)=\sum\limits_{d|n} \alpha_d(x) \beta_{\frac{n}{d}}(x)$
for all $n\ge 1$;
\item[(iii)] $\gamma_n(x)=c_n \alpha_n(x)$ for all $n\ge 1$, where the arithmetic function $n \mapsto c_n$ is completely multiplicative (that is, 
$c_{mn}=c_m c_n$ for all positive integers $m$ and $n$). 
\end{itemize}

Next we show that there is yet another (and much  less obvious) way of constructing new Dirichlet convolution polynomials from existing ones. The following theorem presents a generalization of the construction given in \eqref{def_b_n_pols}. 

\begin{theorem}\label{thm1}
For any Dirichlet convolution polynomials $\alpha_n(x)$ and fixed $w\in {\mathbb C}$ the functions 
\begin{equation}\label{def_beta_n}
\beta_n(x):=\frac{x}{x+w \ln(n)} \alpha_n(x+w\ln(n)), \;\;\; n\ge 1,
\end{equation}
are also Dirichlet convolution polynomials. 
\end{theorem}

Before we prove Theorem \ref{thm1}, we would like to remark that formula \eqref{def_beta_n} implies that $\beta_1(x)=\alpha_1(x)=1$ for all $x$ and for $n\ge 2$ the  Dirichlet convolution polynomials $\beta_n(x)$ can be defined in a more concise way using the hat-operation introduced in \eqref{def_hat_alpha_n}: 
\begin{equation}\label{def_hat_beta_n}
\hat \beta_n(x)=\hat \alpha_n(x+w\ln(n)).  
\end{equation}

\vspace{0.25cm}
\noindent
{\bf Proof of Theorem \ref{thm1}:}
Let $p_1=2, p_2=3, p_3=5, \dots, p_k$ be the first $k$ prime numbers. For a $k$-tuple of non-negative integers ${\bf m}=(m_1,m_2,\dots,m_k)$  we define 
\begin{equation}\label{def_F_n}
F_{{\bf m}}(x):=\alpha_n(x),
\end{equation}
where 
\begin{equation}\label{n_prime_factorization}
n=p_1^{m_1}p_2^{m_2}\cdots p_k^{m_k}. 
\end{equation}
The Dirichlet convolution identity \eqref{alpha_n_mult_convolution} implies the following additive convolution identity for polynomials $F_{\bf m}(x)$: 
\begin{equation}\label{multinomial_convolution}
F_{{\bf m}}(x+y)=\sum\limits_{{\bf 0} \le {\bf j} \le {\bf m}} 
F_{{\bf j}}(x) F_{{\bf m}-{\bf j}}(y),
\end{equation}
where the summation is over all $k$-tuples of integers ${\bf j}=(j_1,j_2,\dots,j_k)$ with $0\le j_l\le m_l$.
Note that the degree of $F_{\bf m}(x)$ is at most $m_1+m_2+\dots+m_k$ (see Proposition \ref{proposition_uniqueness}). This result and the convolution identity \eqref{multinomial_convolution} imply that $F_{\bf m}(x)$ are {\it multinomial convolution polynomials}, as defined by Zeng in \cite{Zeng1996}. Theorem 2 in \cite{Zeng1996} tells us that for any ${\bf t}\in {\mathbb C}^k$ we have 
\begin{equation}\label{derived_convolution_identity}
\frac{(x+y)F_{\bf m}(x+y+{\bf t}\cdot {\bf m})}{
x+y+{\bf t}\cdot {\bf m}}=\sum\limits_{{\bf 0} \le {\bf j} \le {\bf m}} 
\frac{xF_{\bf j}(x+{\bf t}\cdot {\bf j})}{
x+{\bf t}\cdot {\bf j}} \times 
\frac{yF_{{\bf m}-{\bf j}}(y+{\bf  t}\cdot ({\bf m}-{\bf j}))}{
y+{\bf  t}\cdot ({\bf m}-{\bf j})}.
\end{equation}
Here ${\bf t}\cdot {\bf j}$ denotes the dot-product $t_1j_1+t_2j_2+\dots+t_k j_k$. We take 
$$
{\bf t}=w (\ln(p_1),\ln(p_2),\dots,\ln(p_k)),
$$ 
and note that if $n$ is given by \eqref{n_prime_factorization} then 
$$
{\bf t}\cdot {\bf m}=w \ln(n). 
$$
Using this result, formulas \eqref{def_F_n} and \eqref{n_prime_factorization} 
and identity \eqref{derived_convolution_identity} we find that 
for every $n$ that has only prime factors $p_1,p_2,\dots,p_k$ we have
$$
\frac{(x+y)\alpha_n(x+y+w\ln(n))}{x+y+w\ln(n)}=
\sum\limits_{d|n} \frac{x\alpha_d(x+w\ln(d))}{x+w\ln(d)} \times 
\frac{y\alpha_{\frac{n}{d}}(y+w\ln(\frac{n}{d}))}{y+w\ln(\frac{n}{d})}. 
$$
Since $k$ was arbitrary, we conclude that the above Dirichlet convolution identity holds for all $n\ge 1$. In other words, the polynomials $\beta_n(x)$ defined via \eqref{def_beta_n} are Dirichlet convolution polynomials. 
\qed


\section{Inversion theorem for Dirichlet series}\label{section_inversion}


For $\theta\in {\mathbb R}$ we denote
${\mathbb C}_{\theta}:=\{ z\in {\mathbb C} \; : \; \re(z)>\theta\}$. We will also denote by ${\mathcal D}$ the set of Dirichlet series of the form $\sum_{n\ge 1} c_n n^{-s}$, which  converge absolutely in some half-plane ${\mathbb C}_{\theta}$, and by ${\mathcal D}_0$
the set of Dirichlet series in ${\mathcal D}$ that have zero constant term (that is, $c_1=0$).

The following theorem is our second main result. It should be considered a Dirichlet series analogue of the Lagrange Inversion Theorem as presented in formulas \eqref{def_B_z_w}-\eqref{def_a_n_x}. This theorem also extends our earlier results in \cite{Kuznetsov}. 
We recall that the hat-operation is defined in \eqref{def_hat_alpha_n}. 
\begin{theorem}\label{thm_functional_eqn}
Fix a Dirichlet series $f \in {\mathcal D}_0$ and $w\in {\mathbb C}$. 
\begin{itemize}
\item[(i)]
The equation 
$f(s-w g(s))=g(s)$ has a unique solution $g \in {\mathcal D}_0$. 
\item[(ii)] Let $\alpha_n(x)$ be the family of Dirichlet convolution polynomials generated by $f$ via 
$$
e^{x f(s)}=\sum\limits_{n\ge 1} \frac{\alpha_n(x)}{n^s}.
$$ 
Then the solution to the equation $f(s-wg(s))=g(s)$ is given by
\begin{equation}\label{def_hat_g}
g(s)=\sum\limits_{n\ge 2} \frac{\hat \alpha_n(w\ln(n))}{n^s}
\end{equation}
and it satisfies
\begin{equation}\label{hat_g_formula2}
e^{x g(s)}=1+x\sum\limits_{n\ge 2}   \frac{\hat \alpha_n(x+w\ln(n))}{n^s}, \;\;\; x\in {\mathbb C}.
\end{equation}
The Dirichlet series in \eqref{def_hat_g} and \eqref{hat_g_formula2} converge absolutely in some half-plane ${\mathbb C}_{\theta}$. 
\end{itemize}
\end{theorem}

\vspace{0.25cm}
We will begin proving Theorem \ref{thm_functional_eqn} after we make several remarks. 

\vspace{0.25cm}
First of all, we would like to point out that the condition that the Dirichlet series $f$ has zero constant term $c_1$ is not  restrictive. Indeed, assume that we want to solve equation $f(s+g(s))=g(s)$ where $f =\sum_{n\ge 1} c_n n^{-s}$ belongs to ${\mathcal D}$ and $c_1\neq 0$ (so that $f \notin {\mathcal D}_0$ and Theorem \ref{thm_functional_eqn} is not directly applicable). We construct the function $F(s):= f(s-wc_1)-c_1$ that clearly belongs to ${\mathcal D}_0$. Now we can apply Theorem \ref{thm_functional_eqn} and 
find $G$ that solves the equation $F(s-w  G(s))= G(s)$, and then we can recover the solution  to the original equation $f(s-w  g(s))= g(s)$ via 
$g(s)=c_1 +  G(s)$.

Second, comparing 
\eqref{def_hat_g}, \eqref{hat_g_formula2} and \eqref{def_beta_n}, we see that $g$ is the generating Dirichlet series for the Dirichlet convolution polynomials 
\begin{equation}\label{def_beta_again}
\beta_n(x)= \frac{x}{x+w \ln(n)} \alpha_n(x+w\ln(n)). 
\end{equation}
In particular, we could write formulas \eqref{def_hat_g} and \eqref{hat_g_formula2} in the form
$$
g(s)=\sum\limits_{n\ge 2} \frac{\hat \beta_n(0)}{ n^{s}}, \;\;\;  e^{ xg(s)}=\sum\limits_{n\ge 1} \frac{\beta_n(x)}{ n^{s}}.
$$

Finally, the functions $f$ and $g$ in Theorem \ref{thm_functional_eqn} also satisfy the functional equation
\begin{equation}\label{functional_equation2}
g(s+w  f(s))=f(s). 
\end{equation}
This result follows at once by applying Theorem \ref{thm_functional_eqn} to Dirichlet series $g$ and Dirichlet polynomials $\beta_n(x)$ 
and changing $w\mapsto -w$. Note that if $\beta_n(x)$ is defined as in \eqref{def_beta_again}, then 
$$
\frac{x}{x-w \ln(n)} \beta_n(x-w\ln(n))=\alpha_n(x). 
$$

Before we are ready to prove Theorem \ref{thm_functional_eqn}, we need to establish several auxiliary results.

\begin{lemma}\label{lemma1}
For every $f \in {\mathcal D}_0$ the equation $f(s+g(s))=g(s)$ has a unique solution $g \in {\mathcal D}_0$. 
\end{lemma}
\begin{proof}
Let  ${\mathcal H}^{\infty}$ be the Banach space of Dirichlet series $f(s)=\sum_{n\ge 1} c_n n^{-s}$ that converge and are bounded on 
${\mathbb C}_0$, endowed with the norm   
$$
||f||_{\infty}=\sup \{ |f(s)| \; : s \in {\mathbb C}_0 \}.
$$
See \cite{Bayart,Choi} for properties of ${\mathcal H}^{\infty}$ and related Banach spaces. 

In the proof we will need the following result about composition of Dirichlet series. For the proof of this result see \cite{Bayart} (Theorem 11 and remark after Corollary 2).  See also \cite{gordon1999} for characterization of composition operators acting on a certain Hilbert space of Dirichlet series. 
 
\vspace{0.25cm}
\noindent 
{\bf Fact 1:} Let $\phi$ be a Dirichlet series in ${\mathcal D}$ that has an analytic extension 
to ${\mathbb C}_0$ and satisfies $s+\phi(s) \in {\mathbb C}_0$ for all $s \in {\mathbb C}_0$. Then for every 
$f \in {\mathcal H}^{\infty}$ the function 
$f(s+\phi(s))$ also belongs to ${\mathcal H}^{\infty}$.

\vspace{0.25cm} 

The condition $f\in {\mathcal D}_0$ implies that the Dirichlet series $f$ converges absolutely in some half-plane ${\mathbb C}_{\tilde \theta}$ and satisfies
$f(s)=O(2^{-\re(s)})$ and $f'(s)=O(2^{-\re(s)})$ as $\re(s)\to +\infty$.  
Thus we can find $\theta>1$ large enough such that the Dirichlet series  $f$ converges absolutely in ${\mathbb C_{\theta}}$ and satisfies $|f(s)|<1$ and $|f'(s)|<1/2$ for $s \in {\mathbb C}_{\theta}$.  

Denote $G(s)=g(s+2\theta)$ and $F(s)=f(s+\theta)$. Functions $f$ and $g$ satisfy $f(s+g(s))=g(s)$ for $s\in {\mathbb C}_{2\theta}$ if and only if functions
$F$ and $G$ satisfy
\begin{equation}\label{new_eqn}
F(s+\theta+G(s))=G(s)
\end{equation}
for $s \in {\mathbb C}_0$. 
Our first goal is to prove that the equation \eqref{new_eqn} has a solution $G \in {\mathcal H}^{\infty}$.

For $s\in {\mathbb C}_0$ we define $G_0(s)=F(s+\theta)$ and 
$$
G_{k}(s)=F(s+\theta+G_{k-1}(s)), \;\;\; k\ge 1.
$$ 
Due to our choice of $\theta$ above, both functions $F$ and $F'$ lie in  ${\mathcal H}^{\infty}$ and 
satisfy
$|| F ||_{\infty}\le 1$ and $||F'||_{\infty}\le 1/2$. This implies that  $G_0 \in {\mathcal H}^{\infty}$ and 
$||G_0||_{\infty}\le 1$. 
We will prove by induction that the same conditions hold for all $G_k$ with $k\ge 1$. Assume that for some $k$ we have 
$G_k \in  {\mathcal H}^{\infty}$ and $||G_k||_{\infty}\le 1$.
Then the function $s\mapsto s+\theta + G_{k}(s)$ maps ${\mathbb C}_0$ into ${\mathbb C}_0$ (since $|\re(G_k(s))|\le |G_k(s)|\le 1<\theta$ for $s\in {\mathbb C}_0$, thus $\re(\theta+G_k(s))>0$). We can apply Fact 1 and conclude that $G_{k+1} \in {\mathcal H}^{\infty}$. We also have 
$||G_{k+1}||_{\infty}\le 1$, since for all $s\in {\mathbb C}_0$ we have $s+\theta + G_{k}(s)\in {\mathbb C}_0$ and
$$
|G_{k+1}(s)|=|F(s+\theta+G_{k}(s))|\le ||F||_{\infty}\le 1. 
$$

Next, we write for $s\in {\mathbb C}_0$
\begin{equation}\label{g_k_bound1}
G_{k+1}(s)-G_k(s)=F(s+\theta+G_k(s))-F(s+\theta+G_{k-1}(s))=\int\limits_{[z_1,z_2]} F'(z) {\textrm{d}} z,
\end{equation}
where $z_1=s+\theta+G_{k-1}(s)$, $z_2=s+\theta+G_{k}(s)$ and the integration is over the contour
$$
[z_1,z_2]:=\{z=t z_1+(1-t) z_2, \;\;\; 0\le t \le 1\},
$$ 
which is just the straight line interval connecting points $z_1$ and $z_2$. Since $z_1$ and $z_2$ both lie in the half-plane ${\mathbb C}_0$, we have $[z_1,z_2] \subset {\mathbb C}_0$ and we conclude that
\begin{equation}\label{g_k_bound2}
\Big | \int_{z_1}^{z_2} F'(z) {\textrm{d}} z \Big| \le \max\{ |F'(z)| \; : \; z\in [z_1,z_2]\} \times |z_2-z_1|\le 
||F'||_{\infty} \times |z_2-z_1|\le \frac{1}{2} |z_2-z_1|.  
\end{equation}
Since $z_2-z_1=G_k(s)-G_{k-1}(s)$ we obtain from \eqref{g_k_bound1} and \eqref{g_k_bound2}
$$
|G_{k+1}(s)-G_k(s)|\le  \frac{1}{2} |G_k(s)-G_{k-1}(s)|,
$$ 
for all $s\in {\mathbb C}_0$. 
Thus 
$$
||G_{k+1}-G_k||_{\infty}\le  \frac{1}{2} ||G_k-G_{k-1}||_{\infty}
$$
 and $\{G_k\}_{k\ge 0}$ form a Cauchy sequence in ${\mathcal H}^{\infty}$. Thus the limit $G(s)=\lim_{k\to \infty} G_k(s)$ exists and $G \in {\mathcal H}^{\infty}$. For each $s\in {\mathbb C}_0$ we take the limit as $k\to +\infty$ in the identity  $G_{k}(s)=F(z+\theta+G_{k-1}(s))$ and obtain 
$$
G(s)=F(s+\theta+G(s)), \;\;\; \re(s)>0. 
$$
As we discussed above, the existence of a solution $G \in {\mathcal H}^{\infty}$ of the above equation implies the existence of the Dirichlet series 
$g$ that converges in ${\mathbb C}_{2\theta}$ and solves equation $f(s+g(s))=g(s)$. Since the Dirichlet series $g$ converges in ${\mathbb C}_{2\theta}$, it converges absolutely in 
${\mathbb C}_{2\theta+1}$ and we  conclude that $g\in {\mathcal D}$. Since $f(s) \to 0$ as $\re(s)\to  +\infty$ it is also true that 
$g(s)=f(s+g(s)) \to 0$ as $\re(s) \to +\infty$, thus $g \in {\mathcal D}_0$. This ends the proof of existence (in the space 
${\mathcal D}_0$) of a solution $g$ to the equation  $f(s+g(s))=g(s)$.

Now we will prove uniqueness of the solution. Assume that there exists another function $\tilde g \in {\mathcal D}_0$ that also
satisfies $f(s+\tilde g(s))=\tilde g(s)$. Since $\tilde g \in {\mathcal D}_0$, the Dirichlet series for $\tilde g$ converges absolutely for all $\re(s)$ large enough and we have $|\tilde g(s)|<1$ for all $\re(s)$ large enough. 
In particular, for all $\re(s)$ large enough we have $s+\tilde g(s) \in {\mathbb C}_{\theta}$ and $s+g(s) \in {\mathbb C}_{\theta}$. Recall that $|f'(s)|<1/2$ for $s\in {\mathbb C}_{\theta}$. Using the same method of estimate as we used in \eqref{g_k_bound1} and \eqref{g_k_bound2} we obtain
for all $\re(s)$ large enough
$$
|g(s)-\tilde g(s)|=|f(s+ g(s))-f(s+\tilde g(s))|\le \sup\{ |f'(z)| \; : \; z\in {\mathbb C}_{\theta}\} \times |g(s)-\tilde g(s)|\le \frac{1}{2} |g(s)-\tilde g(s)|. 
$$ 
Therefore $\tilde g(s)=g(s)$ for all $\re(s)$ large enough and the solution
$g$ to the equation $f(s+g(s))=g(s)$ is unique in ${\mathcal D}_0$. 
\end{proof}

Everywhere in this paper we will write $f(s)=o(m^{-s})$ if $f$ is defined in some half-plane ${\mathbb C}_{\theta}$ and satisfies 
$|f(s)|m^{\re(s)} \to 0$ as $\re(s)\to +\infty$. Note that a Dirichlet series 
$\phi(s)=\sum_{n\ge 1} c_n n^{-s} \in {\mathcal D}$ satisfies $\phi(s)=o(N^{-s})$ 
if and only if $c_i=0$ for $1\le i \le N$.

\begin{lemma}\label{lemma3}
Let $\alpha_n(x)$ be Dirichlet convolution polynomials.
Fix $w\in {\mathbb C}$, an integer $N\ge 2$ and define
\begin{equation}\label{def_hat_f_N_hat_g_N}
f_N(s):=\sum\limits_{n=2}^N \frac{\hat \alpha_n(0)}{n^s}, \;\;\; 
g_N(s):=\sum\limits_{n=2}^N \frac{\hat \alpha_n(w\ln(n))}{n^s}. 
\end{equation}
Then $f_N(s-w  g_N(s))-g_N(s)=o(N^{-s})$.
\end{lemma}
\begin{proof}
Let $\beta_n(x)$ be the Dirichlet convolution polynomials defined via 
\eqref{def_beta_n}. Then $\hat \beta_n(0)=\hat \alpha_n(w\ln(n))$ and applying Proposition \ref{proposition_properties_alpha_n}(iii) 
to Dirichlet convolution polynomials $\beta_n(x)$ we obtain 
\begin{equation*}
e^{x  g_N(s)}=\sum\limits_{m=1}^N \frac{\beta_m(x)}{m^s} + o(N^{-s}). 
\end{equation*}
Using this fact we compute 
\begin{align}\label{lemma3_proof1}
f_N(s-w  g_N(s))&=\sum\limits_{k=2}^N \frac{\hat \alpha_k(0)}{k^s} e^{w \ln(k)  g_N(s)}
= \sum\limits_{k=2}^N \frac{\hat \alpha_k(0)}{k^s} \sum\limits_{m=2}^N \frac{\beta_m(w\ln(k))}{m^s}+o(N^{-s})\\
\nonumber
&=\sum\limits_{n=2}^N  \Big[ \sum\limits_{d|n, d\ge 2} \hat\alpha_d(0) \beta_{\frac{n}{d}}(w \ln(d)) \Big] \times \frac{1}{n^{s}} +o(N^{-s}).  
\end{align}
Formula \eqref{def_beta_n} implies
$$
\beta_{\frac{n}{d}}(w\ln(d))=\frac{w\ln(d)}{w\ln(d) + w\ln(\frac{n}{d})} \alpha_{\frac{n}{d}}(w\ln(d) + w\ln(\frac{n}{d}))=\frac{\ln(d)}{\ln(n)} \alpha_{\frac{n}{d}}(w\ln(n)).
$$
Using this result and identity \eqref{convolution_eqn} we simplify the expression in the square brackets in \eqref{lemma3_proof1}
\begin{align*}
\sum\limits_{d|n, d\ge 2}\hat\alpha_d(0) \beta_{\frac{n}{d}}(w\ln(d))=
\frac{1}{\ln(n)}\sum\limits_{d|n, d\ge 2}\ln(d)\hat\alpha_d(0) \alpha_{\frac{n}{d}}(w\ln(n))
=\hat\alpha_n(w \ln(n)).
\end{align*}
The above formula and equation \eqref{lemma3_proof1} give us the desired result: 
$$
f_N(s-w g_N(s))=\sum\limits_{n=2}^N \frac{\hat \alpha_n(w \ln(n))}{n^{s}}+o(N^{-s})=
 g_N(s)+o(N^{-s}). 
$$
\end{proof}

\vspace{0.25cm}
\noindent
{\bf Proof of Theorem \ref{thm_functional_eqn}:}
Assume that $w\neq 0$ (otherwise the functional equation $f(s-wg(s))=g(s)$ becomes a trivial equation $f(s)=g(s)$ and there is nothing to prove). 
We want to find  $g(s)$ that solves the equation $f(s-w g(s))=g(s)$. An equivalent problem is to 
 find $G(s):=-wg(s)$ that satisfies the equation $F(s+G(s))=g(s)$, where $F(s):=-w f(s)$. Applying Lemma \ref{lemma1} to this latter equation we see that such a solution $G$ exists and is unique in ${\mathcal D}_0$. Therefore the solution $g(s)$ to $f(s-wg(s))=g(s)$ exists and is unique in ${\mathcal D}_0$. 
 
Next, we fix a large positive integer $N$ and define $f_N$ and $g_N$ as in \eqref{def_hat_f_N_hat_g_N}. Our goal is to show that 
$g(s)- g_N(s)=o(N^{-s})$: this would imply that the first $N$ coefficients in the Dirichlet series of $g$ coincide with the first $N$ coefficients in the Dirichlet series of $g_N$. Since $N$ is arbitrary, this would prove identity \eqref{def_hat_g}.  

We choose $\theta \in \r $ large enough such that the following conditions hold
\begin{itemize}
\item[(i)] $f$ converges absolutely in ${\mathbb C}_{\theta-1}$;
\item[(ii)] $|w  g(s)|<1$ and $|w g_N(s)|<1$ for $s\in {\mathbb C}_{\theta}$;
\item[(iii)] $|w f'(s)|<1/2$ for $s\in {\mathbb C}_{\theta-1}$.
\end{itemize} 

First of all, we note that $f_N(s)-f(s)=o(N^{-s})$. Due to condition (ii) above we have  
$$
f_N(s-w  g_N(s))= f(s-w  g_N(s))+o(N^{-s}). 
$$
Thus we have two equations 
\begin{align*}
g(s)&= f(s-w  g(s)), \\
g_N(s)&= f(s-w  g_N(s)) + o(N^{-s}), 
\end{align*}
that hold for $s\in {\mathbb C}_{\theta}$. 
Subtracting the second equation from the first we obtain
\begin{equation*}
g(s)-g_N(s)=f(s-w  g(s))- f(s-w  g_N(s))+o(N^{-s}), \;\;\; s \in {\mathbb C}_{\theta}. 
\end{equation*}
 Condition (ii) above implies that for $s\in {\mathbb C}_{\theta}$ both points $z_1=s-w g(s)$ and $z_2=s-w g_N(s)$ lie in 
${\mathbb C}_{\theta-1}$, thus we can apply the same method of estimate as we used in \eqref{g_k_bound1} and \eqref{g_k_bound2} and obtain
\begin{align*}
|g(s)-g_N(s)|&\le |f(s-w  g(s))- f(s-w  g_N(s))|+|o(N^{-s})|  \\
& \le \sup\{ |f'(z)| \; : \; z\in {\mathbb C}_{\theta-1}\} \times |w| \times  |g(s)- g_N(s)|+o(N^{-s})\\
& \le \frac{1}{2} |g(s)- g_N(s)|+o(N^{-s}).
\end{align*}
The above inequality implies $g(s)-g_N(s)=o(N^{-s})$.  This ends the proof of 
formula \eqref{def_hat_g}. Formula \eqref{hat_g_formula2} follows at once from \eqref{def_hat_g} and Theorem \ref{thm1}. 
\qed

\vspace{0.25cm}

We would like to point out that Lagrange Inversion Theorem for power series follows from 
Theorem \ref{thm_functional_eqn}. Indeed, starting with a power series $A(z)=1+\sum_{k\ge 2} c_k z^k$ we construct convolution polynomials $a_k(x)$ via 
\begin{equation}\label{last_formula}
A(z)^x=\sum\limits_{k\ge 1} a_k(x) z^k. 
\end{equation}
Changing variables $z=2^{-s}$ gives us a Dirichlet series $A(2^{-s})$ and we define $f(s)=\ln(A(2^{-s}))$. Note that $A$ is analytic near $z=0$ if and only if $f \in {\mathcal D}_0$.  From \eqref{last_formula} we  find  
$$ 
e^{x f(s)}=\sum\limits_{k\ge 1} \frac{a_k(x)}{2^{ks}},
$$
thus the Dirichlet convolution polynomials generated by $f$ are   
\begin{align*}
\alpha_n(x)=
\begin{cases}
a_k(x), \;\;\; {\textrm {if }} \; n=2^k, \\
0, \;\;\qquad {\textrm {otherwise}}. 
\end{cases}
\end{align*} 
Applying Theorem \ref{thm_functional_eqn} to $f$ and then changing variables from $s$ back to $z=2^{-s}$ would give us the Lagrange inversion theorem as presented in formulas \eqref{def_B_z_w}-\eqref{def_a_n_x}.

Next we turn our attention to the question of determining the abscissa of convergence of the Dirichlet series $g$ that solves the equation 
$f(s-w g(s))=g(s)$. In the next result we provide a partial answer to this problem: we focus on the simple case when the Dirichlet series 
$f$ has positive coefficients and $w>0$. For a Dirichlet series $h$ we denote  its abscissa of absolute convergence by $\sigma^h_a$. 

\begin{proposition}\label{prop_abscissa}
 Assume the Dirichlet series $f=\sum_{n\ge 2} c_n n^{-s} \in {\mathcal D}_0$ has nonnegative coefficients $c_n$ and $\sigma^f_a>-\infty$. Fix $w>0$ and let $g \in {\mathcal D}_0$ be the solution of the equation $f(s-wg(s))=g(s)$. Then the Dirichlet series $g$ also has nonnegative coefficients and has abscissa of absolute convergence
\begin{equation}
\sigma_a^g=\min\limits_{s \ge \sigma_a^f} (s + w  f(s)).  
\end{equation}
In particular,  
\begin{itemize}
\item[(i)] if $f'(\sigma_a^f+)<-1/w$ then $\sigma_a^g=s_0+w f(s_0)$ where $s_0>\sigma_a^f$ is the unique solution of $f'(s)=-1/w$;
\item[(ii)] If $f'(\sigma_a^f+)\ge -1/w$ then $\sigma_a^g=\sigma_a^f+wf(\sigma_a^f)$. 
\end{itemize}
\end{proposition}
\begin{proof}
According to Theorem \ref{thm_functional_eqn}, the unique solution in ${\mathcal D}_0$ to the equation $f(s-wg(s))=g(s)$ exists and is given by Dirichlet series \eqref{def_hat_g}. If the Dirichlet series $f$ has positive coefficients, then the Dirichlet convolution polynomials 
$\alpha_n(x)$ generated by $f$ also have positive coefficients (this follows from formula \eqref{alpha_n_formula2}, where $\hat \alpha_n(0)$ should be replaced by $c_n$), thus $\hat \alpha_n(x)>0$ for all $x>0$ and \eqref{def_hat_g} implies that $g$ is also a Dirichlet series with positive coefficients. The rest of the proof is based on Landau's Theorem, which  tells us that the abscissa of convergence of $g$ coincides with the largest point $s^* \in \r$ where $g(s)$ is not analytic.

\begin{figure}
\begin{subfigure}{.5\textwidth}
  \centering
  \includegraphics[width=\linewidth]{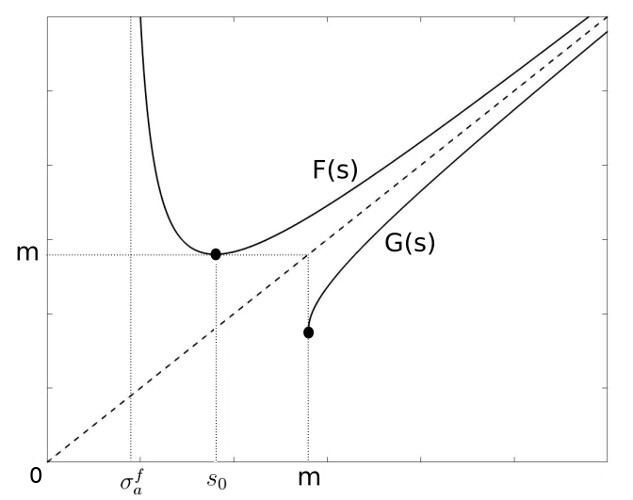}
  \caption{The case $f'(\sigma_a^f+)<-1/w$.}
  \label{Figure1_a}
\end{subfigure}
\begin{subfigure}{.5\textwidth}
  \centering
  \includegraphics[width=\linewidth]{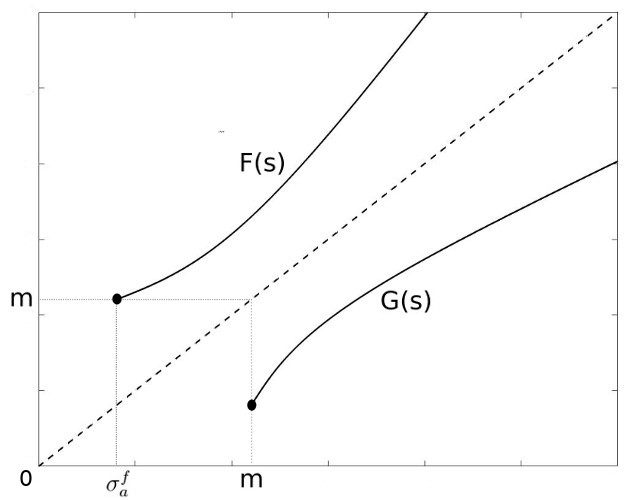}
  \caption{The case $f'(\sigma_a^f+)\ge -1/w$.}
  \label{Figure1_b}
\end{subfigure}
\caption{Illustration to the proof of Proposition \ref{prop_abscissa}.}
\label{fig2}
\end{figure}

We introduce a new function $G(s):=s-wg(s)$, so that the functional equation $f(s-wg(s))=g(s)$ becomes $f(G(s))=(s-G(s))/w$, which we rewrite in an equivalent form $G(s)+w f(G(s))=s$. Thus the function $G(s)$ is the inverse function of $F(s):=s+w f(s)$. Next we collect some properties of $F$:
\begin{itemize}
\item[(i)] $F$ is analytic in the half-plane ${\mathbb C}_{\sigma_a^f}$;
\item[(ii)] $F$ is real-valued and convex on $(\sigma_a^f,\infty)$;
\item[(iii)] $F(s)\to +\infty$ as $s\to +\infty$. 
\end{itemize}
Properties (i) and (iii) follow from the definition of $F$. Property (ii) follows from the fact that 
$$
F''(s)=w f''(s)=\sum\limits_{n\ge 2} \ln(n)^2 \frac{c_n}{n^s} >0, \;\;\; s> \sigma_a^f. 
$$
From the above properties of $F$ we conclude that there exists a unique $s_0 \in (\sigma_a^f, \infty)$ where $F$ achieves its absolute minimum 
$$
F(s_0)=m:=\min\limits_{s \ge \sigma_a^f} F(s). 
$$
Since $F$ is convex and it achieves its minimum at $s_0$, we know that $F$ is strictly increasing on $(s_0,\infty)$ and $F$ maps 
$(s_0,\infty)$ to $(m,\infty)$.   Thus there exists an inverse function $G$ that is defined on $(m,\infty)$, see Figure \ref{fig2}. 
We also know that $F$ is analytic in the neighborhood of any point $s_1>s_0$ and $F'(s_1)>0$, thus the inverse function $G$ is analytic in the neighborhood of $F(s_1)$. This shows that $G(s)$ (and thus $g(s)=(s-G(s))/w$) is analytic near every point $s>m$, thus we can apply Landau's Theorem and conclude that $\sigma_a^g \le m$. 

To prove that $\sigma_a^g=m$ it remains to show that the function $g(s)$ (and thus $G(s)$) is not analytic at $s=m$.  Let us consider frist the 
case when the minimum of $F$ is achieved in $(\sigma_a^f, \infty)$, equivalently, the case when $s_0>\sigma_a^f$ (see Figure \eqref{Figure1_a}). It is clear that $s_0>\sigma_a^f$ if and only if $F'(\sigma_a^f+)<0$, which is equivalent to $f'(\sigma_a^f+)<-1/w$. The function $F$ is analytic near $s_0$ but since $F'(s_0)=0$ and $F''(s_0)>0$ the inverse function $G(s)=F^{-1}(s)$ will have a branching singularity  
at $m=F(s_0)$, thus $\sigma_a^g=m$ in this case. 

Next, let us consider the remaining case when the minimum of $F$ is achieved at $s_0=\sigma_a^f$ (see Figure \ref{Figure1_b}). This case is possible if and only if $F'(\sigma_a^f+)\ge 0$, which is equivalent 
$f'(\sigma_a^f+)\ge -1/w$. Thus the Dirichlet series for $f'(s)$ (which has only nonpositive coefficients) converges absolutely at 
$s=\sigma_s^f$, and this implies that the Dirichlet series for $f(s)$ also converges absolutely at $s=\sigma_a^f$. If $f'(\sigma_a^f+)=-1/w$ (equivalently, if $F'(\sigma_a^f+)=0$) the inverse function $G$ would satisfy $G(m+)=+\infty$, thus it can't be analytic at $m$. If 
$f'(\sigma_a^f+)>-1/w$ (equivalently, if $F'(\sigma_a^f+)>0$), assuming that $G$ is analytic near $m$ we would have $G'(m)>0$, so that the inverse function $F(s)=G^{-1}(s)$ should be analytic in the neighbourhood of $s=\sigma_a^f$, which we know is not the case (again, due to Landau's Theorem, the function $f$ is not analytic at the abscissa of absolute convergence). Thus we arrive at a contradiction and we conclude that $\sigma_a^g=m$. 
\end{proof}

\paragraph{Acknowledgements}
The research was supported by the
Natural Sciences and Engineering Research Council of Canada.




\end{document}